\documentclass[12pt]{article}
\usepackage{etex}
\usepackage[title,titletoc,toc]{appendix}
\usepackage{amsmath,amsthm}

\usepackage{wrapfig}
\usepackage{caption}
\usepackage{lscape}
\usepackage{amssymb}
\usepackage{threeparttable}
\usepackage{eucal}
\usepackage{mathrsfs}
\usepackage{url}
\usepackage[all]{xy}
\usepackage{pdflscape}
\usepackage{fancyhdr}
\usepackage{pdfpages}
\usepackage{subcaption}

\usepackage{bbold}
\usepackage[T1]{fontenc}
\usepackage{graphicx}
\usepackage{rotating}
\usepackage{float}
\usepackage{booktabs}
\usepackage[unicode=true,bookmarks=true,backref=true,colorlinks=true]{hyperref}

\newtheorem{theorem}{Theorem}
\newtheorem{lemma}[theorem]{Lemma}
\newtheorem{corollary}[theorem]{Corollary}
\newtheorem{proposition}[theorem]{Proposition}

\theoremstyle{remark}
\newtheorem*{remark}{Remark}

\theoremstyle{nonumber}

\usepackage{marginnote}
\usepackage[top=3cm, bottom=3cm, left=2cm, right=5cm, heightrounded, marginparwidth=4cm, marginparsep=0.3cm]{geometry}

\numberwithin{equation}{section}
\numberwithin{theorem}{section}

\newcommand{\lie}[1]{\mathfrak{#1}}

\usepackage{comment}
\usepackage{hyperref}
\usepackage{url}

\usepackage{xcolor}
\begin{document}
\thispagestyle{empty}
\raggedbottom
\title{The Largest Volume Conjugacy Class in Most Compact Simple Lie Groups}
\author{Woody Lichtenstein}
\maketitle

\begin{abstract}
  We provide some details about the largest volume conjugacy class in compact simple Lie groups of types $\mathrm A_n$, $\mathrm B_n$, $\mathrm C_n$, $\mathrm D_n$, and $\mathrm G_2$.
\end{abstract}

\section{Introduction}
Let $G$ be a compact connected simple Lie group with maximal torus $T$.  Every element of $G$ is conjugate to an element of $T$, and the conjugacy classes of $G$ are parametrized by $T/W$ where $W$ is the Weyl group.  The size of each conjugacy class is given by the Weyl Jacobian $J$ in the Weyl Integration formula $$J: T/W \to \mathbb R^+.$$  Equivalently, $J$ is a $W$-invariant function $$J: T \to \mathbb R^+ \text{ (the non-negative real numbers)}$$
vanishing on the singular set of $T$.  Since $T$ is compact, $J$ has a maximum on $T/W$.  For the groups studied in this paper that maximum is unique and can be described by simple geometric or algebraic properties.

\subsection*{Related work}

For the classical matrix groups, finding the largest conjugacy class is the same as determining the most likely set of eigenvalues for a random group element.  This is a tiny step towards understanding more sophisticated questions about the distribution of eigenvalues of random matrices, which is a large well developed subject.  A nice survey with many references is \cite{diaconis}.  In particular, section 4 
of that paper begins with a discussion of the Weyl Jacobian for unitary groups and includes an explanation of the relation between the unitary Weyl Jacobian and Toeplitz determinants.

\section{Measuring the Size of a Conjugacy Class}
Let $\lie g$ be the Lie algebra of $G$, and let $\lie t$ be the Lie algebra of $T$. Under the adjoint action of $T$, $\lie g$ decomposes into
$$
\lie g = \lie t \oplus \sum_{i=1}^m \lie l_i 
$$
where each $\lie l_i$ is a $T$-invariant subspace of real dimension 2. \\

\noindent Restricted to $\lie l_i$, the adjoint action of $\exp(H) \in T$ is rotation through an angle $\alpha_i(H)$ where $H \in \lie t$ and $\alpha_i:\lie i\to \mathbb R$ is a linear function.  [The complexification $\lie l_i\otimes \mathbb C$ splits into two $T$-invariant subspaces of complex dimension 1.  The eigenvalues for $\exp(H)$ on these subspaces are $\exp(\pm \alpha_i(H)\sqrt{-1})$ where $\alpha_i$ is a positive root of $\lie g$ and $-\alpha_i$ is a negative root of $\lie g$.] \\

\noindent With respect to the $\mathrm{Ad}(G)$-invariant definite bilinear form on $\lie g$,  $\lie t^\perp =  \sum_{i=1}^m \lie l_i $.
The $2m$-dimensional volume of the conjugacy class containing $t\in T$ is proportional to the determinant of the derivative of the adjoint action of $G/T$ on $t$.  This is a map from $\lie t^\perp$ to itself that can be computed as
$$\lim_{s \to 0} \left(\exp(sX) t \exp(-sX)\right), \text{ for }X \in \lie t^\perp,$$
 which is just $Xt - tX$.  \\

 \noindent Expressing $Xt-tX$ as right translation by $t$ of an element of $\lie t^\perp$ gives $$Xt-tX = (X - tXt^{-1})t.$$  So we need to compute the determinant of the map $X \mapsto (\mathrm{Id} - \mathrm{Ad}(t)) X$.  This is a product of the determinants of the associated maps from each  $\lie l_i$ to itself.  Each of these associated maps is of the form $\mathrm{Id} - \mathrm{Rotation}(\theta)$ for some angle $\theta$, and that has determinant \\
 
\noindent $\det(\mathrm{Id} - \mathrm{Rotation}(\theta))=(1 - \cos(\theta))^2 + \sin^2(\theta) = 2(1 - \cos(\theta)) = 4\sin^2\left(\frac\theta2\right).$  \\
 
 \noindent Thus the $2m$-dimensional volume of the conjugacy class containing $t\in T$ is proportional to the product over the positive roots $\alpha_i$ of $\sin^2( \alpha_i(\log(t))/2 )$.  Ignoring constant factors,
 \begin{equation}
 \label{volume formula}
 V(t) = \prod_{\alpha \in \Delta}\sin^2( \alpha( \log(t) )/2 ), 
 \end{equation}
where $\Delta$ is the set of positive roots.

\section{The special unitary group}
$SU(n)$ is the group of unitary $n \times n$ matrices with determinant 1.  For $SU(n)$, $T/W$ can be parametrized by diagonal matrices $$t_{\mathbf{\theta}} := \mathrm{diag}(e^{i\theta_0}, \ldots, e^{i\theta_{n-1}})$$
with 
$$0 \leq \theta_0 \leq \cdots \leq \theta_{n-1} \leq 2\pi$$ and $\theta_0 + \cdots  +\theta_{n-1} =  2\pi m$ for some integer $m$.\\

 \noindent The positive roots $\Delta$ have values at $\log(t_\mathbf \theta)$ consisting of the set
 $$\{\theta_k - \theta_j \mid 0 \leq j < k < n\}.$$
 
\noindent The factor $\sin^2( (\theta_k - \theta_j)/2 )$ is half the straight line distance in the complex plane between $e^{i\theta_j}$ and $e^{i\theta_k}$.  So the volume of the conjugacy class corresponding to the diagonal matrix $t_{\mathbf{\theta}}$  attains its maximum when the n-gon inscribed in the unit circle defined by those n vertices maximizes the product of the lengths of all its edges and diagonals.
\begin{theorem}
  The volume of the conjugacy class corresponding to the diagonal matrix $t_\mathbf{\theta}$ is maximized when those $n$ diagonal matrix entries are equally spaced around the unit circle, i.e.\ when $\theta_k= \pi(2k)/n$ for odd $n$, or $\theta_k = \pi(2k+1)/n$ for even $n$.
  \end{theorem}
  %
\begin{proof}
Each positive-negative root pair corresponds to a pair of angles $\theta_j, \theta_k$ with $0 \leq j \neq k < n$.  We choose the root $\theta_k-\theta_j$ for which $r := k-j \mod n$  falls in the interval $[0, n/2]$. \\  

\noindent \textbf{Example:} When $n = 18, j = 4, k = 15$, we pick $\theta_4 - \theta_{15}$  because $9 \geq 7 \equiv 4 - 15 \mod 18$. \\

\noindent Unless $r = n/2$, the root  $\theta_k-\theta_j$ is part of a cycle of length $M = n/\gcd(n,r)$ of roots $$\theta_{j+(m+1)r}-\theta_{j+mr} , \,\,0 \leq m < M, $$where subscript addition is mod $n$. This cycle wraps around the circle $q=r/\gcd(n,r)$ times.  Thus we can decompose the set of positive-negative root pairs into disjoint subsets, and correspondingly we can split the product in the formula for $V(t)$ into factors, according to the residue $r$. The main idea of the proof is that maximizing any one of these factors implies that a subset of the eigenvalues should be evenly spaced. Since these conditions are all compatible, we conclude that the maximum of $V(t)$ occurs where all the disjoint factors are simultaneously maximized, and that occurs where all the eigenvalues are evenly spaced. The values of the evenly spaced eigenvalues are then fixed by the condition that their product should be 1. \\
 
\noindent Now for fixed $j$ and $r$ consider the problem of maximizing 
$$\sin^2((\theta_{j+r}-\theta_j)/2)\cdot \sin^2((\theta_{j+2r}-\theta_{j+r})/2)\cdot \cdots\cdot \sin^2((\theta_{j+Mr}-\theta_{j+(M-1)r})/2)$$ subject to the constraint $$(\theta_{j+r}-\theta_j) + (\theta_{j+2r}-\theta_{j+r}) + \cdots + (\theta_{j+Mr}-\theta_{j+(M-1)r}) = 2\pi q.$$

\noindent Setting $\beta_k = \theta_{j+(k+1)r}-\theta_{j+kr}$, for $k = 0, \ldots, M-1$, this is equivalent to maximizing
 $$
 f(\beta_0,\beta_1,\ldots,\beta_{M-1}) =  \sin^2(\beta_0/2)\sin^2(\beta_1/2).\cdot\cdots\cdot\sin^2(\beta_{M-1}/2)
 $$ 
 subject to the constraint 
 $$g(\beta_0,\beta_1,\ldots,\beta_{M-1}) = \beta_0 + \beta_1 +\cdots+ \beta_{M-1}= 2\pi q.$$
 Replacing $f$ with $\log(f)$, the method of Lagrange multipliers implies that $$(\cot(\beta_0/2), \ldots, \cot(\beta_{M-1}/2) )$$ must be proportional to $(1,\ldots, 1)$, i.e. $$\cot(\beta_0/2) = \cot(\beta_1/2) = \cdots = \cot(\beta_{M-1}/2)$$ or 
  $$\beta_0 = \beta_1 = \cdots = \beta_{M-1}.$$
 In other words, $\theta_j, \theta_{j+r}, \ldots , \theta_{j+(M-1)r}, \theta_{j+Mr}= \theta_j$ must be equally spaced.
The special case of even $n$, with $r = n/2$, is slightly different, because a single positive-negative root pair makes up a cycle of length 2 that wraps once around the circle.  [Example: when $n = 18, r = 9, j = 4$, the cycle consists of $\theta_{13}-\theta_4$ and $\theta_4-\theta_{13}$.]  But the same general principle applies. Set $\beta = \theta_{j+n/2}-\theta_j$.  The maximum of $f(\beta) = \sin^2(\beta/2)$ is 1, and occurs at $\beta = \pi$, i.e.\ when $\theta_j$ and $\theta_{j+n/2}$ are evenly spaced.
\end{proof}

\section{The Even Orthogonal Group}
For $SO(2n)$, $T/W$ can be parametrized by $2\times 2$ block diagonal rotation matrices with rotation angles $0 \leq \theta_1 \leq \cdots \leq \theta_n \leq \pi$.  The positive roots are $\theta_k - \theta_j$ and $\theta_k + \theta_j, 1 \leq j < k \leq n$.

\begin{lemma}\label{diff of cosines}$\sin^2((\theta_1-\theta_2)/2)\sin^2((\theta_1+\theta_2)/2) = (\cos\theta_1 - \cos\theta_2)^2/4$
\end{lemma}

\begin{proof}  Use the half-angle formula for sine and straightforward algebra. 
\end{proof}

\begin{corollary} The largest conjugacy class of $SO(2n)$ corresponds to rotation angles $0 \leq \theta_1 \leq \cdots \leq \theta_n \leq \pi$ for which the real polynomial 
$$p(x) = (x - \cos\theta_1)(x - \cos\theta_2)\cdots(x - \cos\theta_n)$$
has the largest discriminant among all real polynomials of degree $n$ with $n$ real roots in the interval $[-1,1]$.
\end{corollary}

\begin{lemma} \label{pm1roots}If $p(x) = (x - x_1)(x - x_2)\cdots(x - x_n)$ has the largest discriminant among all real polynomials with real roots $1 \geq x_1 > x_2 > \cdots > x_n \geq -1$, then $x_1 = 1$ and $x_n = -1$.
\end{lemma}

\begin{proof}  If $x_1 < 1$, then increasing $x_1$ to 1 increases its distance from all the other roots and hence increases the discriminant.  Similarly if $x_n > -1$, then decreasing $x_n$ to $-1$ increases its distance from all the other roots and hence increases the discriminant.
\end{proof}

\begin{theorem} \label{ODEtypeDn} If $p(x) = (x - x_1)(x - x_2)\cdots(x - x_n)$ has the largest discriminant among all real polynomials with real roots $1 \geq x_1 > x_2 > \cdots > x_n \geq -1$, then p satisfies the ordinary differential equation $$p''(x) = \frac{-n(n-1)}{(1-x^2)}p(x)$$.
\end{theorem}

\begin{proof} Let $X = \{ (x_2, x_3, \ldots, x_{n-2}, x_{n-1} \mid  1=x_1 > x_2 > \cdots > x_n = -1 \}$ with closure $\overline{X}$ and boundary $\overline{X} - X$.  Let $D(x_1,x_2,\ldots,x_n) = \prod_{1 \leq i < k \leq n}(x_i - x_k)$ be the positive square root of the discriminant of $p(x)$.  So $D > 0$ on X, $D \geq 0$ on $\overline{X}$, and $D \equiv 0$ on the boundary $\overline{X} - X$.  Thus the maximum of $D$ on the compact set $\overline{X}$ occurs in $X$ and its minimum occurs on the boundary $\overline{X} - X$.

\bigskip
\noindent For $2 \leq j \leq n-1$, let $$E_j(x_1,x_2,\ldots,x_n) = (x_1-x_j)(x_2-x_j)\cdots(x_{j-1}-x_j)(x_j-x_{j+1})\cdots(x_j-x_n)$$ be the product of the factors of $D$ that include $x_j$, let $$F_j = D/E_j = \prod_{1 \leq (i \neq j) < (k \neq j) \leq n}(x_i - x_k)$$, and 
let $$q_j(x) = (x - x_1)(x - x_2)\cdots(x-x_{j-1})(x-x_{j+1})\cdots(x - x_n) = p(x)/(x-x_j)$$
be the product of the factors of $p$ that do not include $x_j$.  Then up to sign, $q_j’(x_j) = \frac{\partial E_j}{\partial x_j}$.  

\bigskip
\noindent At the maximum of $D$, $\frac{\partial D}{\partial x_j} = 0$, and since $F_j > 0$ on X, $\frac{\partial E_j}{\partial x_j} = 0$.  Thus at the maximum of $D$, $p(x) = (x-x_j)q_j(x)$ with $q_j’(x_j) = 0$.  Now $p’(x) = (x-x_j)q_j’(x) + q_j(x)$, and $p''(x) = (x-x_j)q_j''(x) + 2q_j’(x)$, and therefore $p''(x_j) = 0$.  \\

\noindent This shows that $(1-x^2)p''(x)$ and $p(x)$ have all the same roots, and therefore they agree up to a constant factor.  Since $p(x)$ has highest degree term $x^n$, $p''(x)$ has highest degree term $n(n-1)x^{n-2}$, the constant factor must be $-n(n-1)$.
\end{proof}

\begin{corollary} If $n$ is even, $p(x)$ is even.  If n is odd, $p(x)$ is odd.
\end{corollary}

\begin{proof} Let $ax^{n-(2k+1)}$ be the highest degree term in $p(x)$ with parity opposite to $n$, with $k \geq 0$.  Then $(n-2k-1)(n-2k-2)ax^{n-(2k+3)}$ is the highest degree term in $p''(x)$ with parity opposite to $n$.  Since $p(x) = \frac{-1}{n(n-1)}(1-x^2)p''(x)$, we must have $\frac{-1}{n(n-1)}(n-2k-1)(n-2k-2)a = a$.  But this implies $a=0$, so all terms in $p(x)$ must have degree of the same parity as $n$.
\end{proof}

\begin{corollary}  The cosines of the rotation angles of the largest conjugacy class of $SO(2n)$ are algebraic over $\mathbb Q$.
\end{corollary}

\begin{proof} Assume $p(x) = x^n + a_{n-2}x^{n-2} + a_{n-4}x^{n-4} + \ldots$ . Using the equation $p(x) = \frac{-1}{n(n-1)}(1-x^2)p''(x)$, we can successively solve for $a_{n-2}, a_{n-4}, \ldots $.  At each step we get equations with rational coefficients, e.g. $a_{n-2 }= \frac{(n)(n-1)}{((n-2)(n-3) - (n)(n-1))}$, 
$a_{n-4} = \frac{(n-2)(n-3)a_{n-2}}{(n-4)(n-5) - (n)(n-1)}$, etc.  Thus the coefficients of $p$ are all rational.
\end{proof}

\bigskip
\noindent [The first few instances of $p(x)$ for $n = 1, 2, 3$ are $x$, $x^2-1$, and $x^3 - x = x(x^2 - 1)$ respectively.]
\bigskip

\begin{corollary}  \label{equidistribution}For very large n, the rotation angles of the largest conjugacy class of $SO(2n)$ are close to evenly distributed around the circle.
\end{corollary}

\begin{proof} On an interval of length $\Delta x$ away from the end points $-1$ and 1 and short enough for $1-x^2$ to be approximately constant, 
$p''(x) = -\lambda^2p(x)$ with $\lambda = \sqrt{\frac{n(n-1)}{(1-x^2)}}$, so on this interval $p(x) \approx \sin\lambda x$, and therefore $p(x)$ should have approximately $\lambda\Delta x/\pi$ zeros. Since $x = \cos\theta$, $dx = -\sin\theta d\theta$, and therefore $\Delta x/\sqrt{1-x^2} \approx \Delta\theta$.  It follows that $\lambda\Delta x/\pi \approx \sqrt{(n(n-1))}\Delta\theta/\pi\approx\frac n{\pi} \Delta\theta$.  A more rigorous proof may be found in the Appendix.
\end{proof}

\noindent Remark: Sam Lichtenstein pointed out that the polynomials $p(x)$ defined here satisfy the same ODE as the Jacobi polynomials with parameters $\alpha=-1$ and $\beta=-1$.  A detailed and rigorous treatment of asymptotics for Jacobi polynomials is available in \cite{szego} Chapter 8.  See for example Theorem 8.21.8. \\

\begin{corollary} For $SO(2n)$ there is a unique maximum for $V(t)$ in $T/W$.
\end{corollary}
\begin{proof} The set of cosines of the rotation angles for any maximum of $V(t)$ is  defined by the roots of $p(x)$.  The rotation angles themselves are therefore defined up to order and sign.  In particular, any two choices of signs differ by some number $0 \leq k \leq n-2$ of sign changes of the rotation angles that are not 0 or $\pi$.  The Weyl group of SO(2n) includes all permutations and even numbers of sign changes.  So if $k$ is even we know that the two choices of an element of $T$ are conjugate and therefore identical in $T/W$.  But if $k$ is odd, we can also change the sign of either of the rotation angles that is 0 or $\pi$ without changing the selected element of T, and therefore we can still find an element of W that transforms one of the two selections of angles to the other. 
\end{proof}

\begin{corollary} For $SO(8)$ the cosines of the rotation angles of the largest conjugacy class are $\pm1$, $\pm\sqrt{1/5}$.
\end{corollary}
\begin{proof} For $n=4$, $p(x) = x^4 - (6/5)x^2 + (1/5) = (x^2 - 1)(x^2 - (1/5))$.
\end{proof}

\begin{corollary} For $SO(10)$ the cosines of the rotation angles of the largest conjugacy class are $0$, $\pm1$, $\pm\sqrt{3/7}$.
\end{corollary}
\begin{proof} For $n=5$, $p(x) = x^5 - (10/7)x^3 + (3/7)x = x(x^2 - 1)(x^2 - (3/7))$.
\end{proof}

\begin{corollary}  For $SO(12)$ the cosines of the rotation angles of the largest conjugacy class are $\pm1$, $\pm\sqrt{(1/3)(1\pm\sqrt{4/7})}$.
\end{corollary}
\begin{proof} For $n=6$, \\
$p(x) = x^6 - (5/3)x^4 + (5/7)x^2 - (1/21) = (x^2 - 1)(x^4 - (2/3)x^2 + (1/21))$.
\end{proof}

\section{The Odd Orthogonal Group}

For $SO(2n+1)$, $T/W$ can be parametrized by $2\times2$ block diagonal rotation matrices with rotation angles $0 \leq \theta_1 \leq \cdots \leq \theta_n \leq \pi$.  The positive roots are $\theta_j$, $\theta_k - \theta_j$ and $\theta_k + \theta_j$, $1 \leq j < k \leq n$.  Comparing with $SO(2n)$, the maximal torus $T$ fixes a vector orthogonal to the $n$ rotation planes corresponding to the $2\times 2$ block diagonal rotation matrices.  The roots $\theta_j$ correspond to Lie algebra elements that mix the $j^{th}$ rotation plane with the fixed vector.\\

\noindent For rotation angles $0 \leq \theta_1 \leq \cdots \leq \theta_n \leq \pi$, we continue to denote by $p(x)\in \mathbb{R}[x]$ the polynomial
$$p(x) = (x - \cos\theta_1)(x - \cos\theta_2)\cdots(x - \cos\theta_n),$$
and by $D(x)$ the positive square root of the discriminant of $p$.\\

\noindent Denote by $f$ the function
$$ f(x) := \sqrt{1-x}\cdot p(x).$$

\begin{proposition} The largest conjugacy class in $SO(2n+1)$ corresponds to rotation angles $\theta_i$ for which the function $f(x)$ 
has the largest ``type $\mathrm B_n$'' modified square root discriminant 
$$M(\cos\theta_1,\ldots,\cos\theta_n) := \sqrt{(1-\cos\theta_1)(1-\cos\theta_2)\cdots(1-\cos\theta_n)}D(\cos\theta_1,\ldots,\cos\theta_n).$$
\end{proposition}

\begin{proof}  By Lemma \ref{diff of cosines}, up to constants, the positive square roots of the factors in formula \eqref{volume formula} corresponding to the roots $\theta_k - \theta_j$ and $\theta_k + \theta_j$ comprise exactly $D(\cos\theta_1,\ldots,\cos\theta_n)$.  The remaining factors in the square root of formula \eqref{volume formula} correspond to the roots $\theta_j$.  Since $\sin^2(\psi/2)=(1-\cos\psi)/2$, these agree with $M/D$ (again up to constants).
\end{proof}

\begin{theorem}  \label{ODEtypeBn} When $M$ is at its maximum, $f(x)$ satisfies the ordinary differential equation 
$$f''(x) = \frac{-n^2}{1-x^2}\frac{1-x + (1/(4n^2))(1+x)}{1-x}f(x).$$
\end{theorem}

\begin{proof} Straightforward computation gives 
$$f''(x) = (1-x)^{-3/2}[ (1-x)^2p''(x) - (1-x)p'(x)  -(1/4)p(x) ] = (1-x)^{-3/2}[ z(x) ],$$
where $z(x)$ is a polynomial of degree $n$.\\
 
\noindent As in the proof of Lemma \ref{pm1roots}, we know that $\cos\theta_n = -1$ when $M$ achieves its maximum, so $p(x)$ is divisible by $(x+1)$.  Let $p(x) = (x + 1)q(x)$.\\

\noindent As in the proof of Theorem \ref{ODEtypeDn} we can conclude that for any root $\gamma$ of $q$, $$f(x) = (x - \gamma)g_\gamma(x),$$ where 
$g_\gamma’(\gamma) = 0$.  And again as in the proof of Theorem \ref{ODEtypeDn} it follows that $f''(\gamma) = 0$.  This implies that $z(x)$ must be divisible by $q$, which has degree $n-1$, 
and therefore $z(x) = (\lambda x + \mu)q(x)$, or $$(1+x)z(x) = (\lambda x + \mu)p(x),$$ for suitable constants $\lambda,\mu$. Define $$w(x) = z(x) + \frac14p(x) =  (1-x)^2p''(x) - (1-x)p'(x).$$  Now
\begin{align*}
(1+x)w(x) = (1+x)(z(x) + \frac14 p(x)) & =  (\lambda x + \mu)p(x) +\frac{1}{4}(1+x)p(x) \\
             & = ((\lambda+\frac14)x + (\mu+\frac14))p(x).
 \end{align*}
 Note that if any $\theta_j=0$ then $M$ vanishes, so when $M$ is at its maximum, $x=1$ is not a root of $p$. 
  Since $w(x)$ is divisible by $(1-x)$ while $p(x)$ is not, it follows that $$(\lambda+\frac14)x + (\mu+\frac14)$$ must vanish at $x=1$, and thus we can write
$(1+x)w(x) = \beta(1-x)p(x)$.  Comparing coefficients of $x^{n+1}$ gives $\beta = -n^2$.\\

\noindent Finally, straightforward substitution of $f(x) = \sqrt{1-x}p(x)$ into $$f''(x) = (1-x)^{-3/2}[ w(x) - \frac14p(x) ]$$ yields the formula to be proved.
\end{proof}

\begin{corollary} \label{asymptotic corollary} For very large n, $f''(x) \approx \frac{-n^2}{1-x^2}f(x)$, and therefore for very large n, the rotation angles of the largest conjugacy class of $SO(2n+1)$ are close to evenly distributed around the circle.
\end{corollary}

\begin{proof}  See the proof of Corollary \ref{equidistribution}.
\end{proof}

\begin{corollary} For $SO(7)$ the cosines of the rotation angles of the largest conjugacy class are $-1$, $(1\pm\sqrt{6})/5$.
\end{corollary}

\begin{proof}  Use $w(x) = (-9)(1-x)q(x)$ where $q(x) = x^2 + bx + c$, $p(x) = (1+x)q(x)$, and $w(x) = (1-x)^2p''(x) - (1-x)p'(x)$.  The result is $q(x) = x^2 - (2/5)x - (1/5)$ or $p(x) = x^3 +(3/5)x^2 - (3/5)x - (1/5)$.
\end{proof}

\section{The Symplectic Group}

Recall that the compact symplectic group $Sp(2n)$ is the intersection \\ $SU(2n)\cap Sp(2n, \mathbb C) \subset GL_{2n}(\mathbb C)$, i.e.\ the matrices which preserve both the standard hermitian form and the standard symplectic form on $\mathbb C^{2n}$, see e.g. \cite{murnaghan}.  For $Sp(2n)$ take $T/W$ to be diagonal matrices of the form $e^{i\theta_1}, e^{i\theta_2}, \ldots , e^{i\theta_n},e^{-i\theta_1}, e^{-i\theta_2}, \ldots , e^{-i\theta_n}$ with $0 \leq \theta_1 \leq \cdots \leq \theta_n \leq \pi$.  The positive roots are $2\theta_j$, $\theta_k - \theta_j$ and $\theta_k + \theta_j$, $1 \leq j < k \leq n$.   [The roots $2\theta_j$ correspond to Lie algebra elements that mix the eigenspaces for $e^{i\theta_j}$ and $e^{-i\theta_j}$.]\\

\noindent For rotation angles $0 \leq \theta_1 \leq \cdots \leq \theta_n \leq \pi$, we continue to denote by $p(x)\in \mathbb{R}[x]$ the polynomial
$$p(x) = (x - \cos\theta_1)(x - \cos\theta_2)\cdots(x - \cos\theta_n),$$
and by $D(x)$ the positive square root of the discriminant of $p$.\\

\noindent In this section, we denote by $f=f_{\mathrm{type\ C}_n}$ the function
$$ f(x) := \sqrt{1-x^2}\cdot p(x),$$
and by $M$ the ``type $\mathrm C_n$'' modified square root discriminant
\begin{align*}
M(\cos\theta_1,\ldots,\cos\theta_n) & = \sqrt{(1-\cos^2\theta_1)(1-\cos^2\theta_2)\cdots(1-\cos^2\theta_n)}D(\cos\theta_1,\ldots,\cos\theta_n)\\
                                                        & = \sin \theta_1\cdot \sin\theta_2\cdot\cdots\cdot \sin\theta_n D(\cos\theta_1,\ldots,\cos\theta_n).
\end{align*}

\begin{proposition} The largest conjugacy class in $Sp(2n)$ corresponds to the rotation angles $0 \leq \theta_1 \leq \cdots \leq \theta_n \leq \pi$ 
for which the type $\mathrm C_n$ modified square root discriminant $M$ achieves its maximum.
\end{proposition}

\begin{proof} By Lemma \ref{diff of cosines}, up to constant factors the positive square root of the factors in formula \eqref{volume formula} for the volume of a conjugacy class corresponding to the roots $\theta_k-\theta_j$ and $\theta_k + \theta_j$ are exactly $D(\cos\theta_1,\ldots,\cos\theta_n)$.  The remaining factors in the square root of formula \eqref{volume formula} correspond to the roots $2\theta_j$, which match up with the remaining factors in $M$.
\end{proof}

\begin{theorem}  When $M$ is at its maximum, $f(x)$ satisfies the ordinary differential equation 
$$f''(x) = \frac{-(n^2+n)[(1-x^2) + 1/(n^2+n)]}{(1-x^2)^2}f(x).$$
\end{theorem}

\begin{proof}  The proof is entirely analogous to that of Theorem \ref{ODEtypeBn}, using the fact that in this case $\pm 1$ cannot be roots of $p$.
\end{proof}

\begin{corollary} For very large n, $f''(x) \approx \frac{-(n^2+n)}{1-x^2}f(x)$, and therefore for very large n, the eigenvalues of the largest conjugacy class of $Sp(2n)$ are close to evenly distributed around the circle.
\end{corollary}

\begin{proof}  See the proof of Corollary \ref{asymptotic corollary}.
\end{proof}

\begin{corollary} For $Sp(4)$ the real parts of the eigenvalues of the largest conjugacy class are $\pm\sqrt{1/3}$.
\end{corollary}

\begin{proof}  Use $w(x) = 6(x^2-1)p(x)$ where $w(x) = (1-x^2)^2p''(x) - 2x(1-x^2)p'(x)$ and $p(x) = x^2+b$. The result is $p(x) = x^2 - \frac{1}{3}$.
\end{proof}

\begin{corollary} For $Sp(6)$ the real parts of the eigenvalues of the largest conjugacy class are $0,\pm\sqrt{3/5}$.
\end{corollary}

\begin{proof}  Use $w(x) = 12(x^2-1)p(x)$ where $w(x) = (1-x^2)^2p''(x) - 2x(1-x^2)p'(x)$ and $p(x) = x^3+bx$. The result is $p(x) = x^3 - \frac{3}{5}x$.
\end{proof}

\begin{corollary} For $Sp(8)$ the real parts of the eigenvalues of the largest conjugacy class are 
$\pm((3/7) \pm (4/7)\sqrt{3/10})$.
\end{corollary}

\begin{proof}  Use $w(x) = 20(x^2-1)p(x)$ where $w(x) = (1-x^2)^2p''(x) - 2x(1-x^2)p'(x)$ and $p(x) = x^4 + bx^2+c$. The result is $p(x) = x^4 - \frac{6}{7}x^2 +\frac{6}{70}$.
\end{proof}

\section{The Exceptional Group $\mathrm G_2$}

We describe points $t \in T$ by giving the values of 3 short positive roots at $\log(t)$.  To be specific, let $(\theta_1,\theta_2,\theta_3)$ to be the values of 3 of the 6 short roots, each with angle $\frac{2\pi}{3}$ relative to the other two.  Note that these 3 roots are not linearly independent since they satisfy the relation $\theta_1 + \theta_2 + \theta_3 = 0$.  The 3 overlapping pairs $(\theta_1,\theta_2)$, $(\theta_2,\theta_3)$, $(\theta_3,\theta_1)$ each determines a long/short orthogonal root pair by sum and difference [example: $(\theta_1,\theta_2)$ corresponds to $\rm{sum} = \theta_1+\theta_2 = -\theta_3$ which is short, and $\rm{difference} = \theta_1- \theta_2$ which is long] and the resulting 3 pairs exactly cover all 6 positive/negative root pairs.\\

\begin{remark}
  \cite{g2blog} points out that there are two isomorphic dual representations for the $\lie g_2$ root system inside the 2-plane $x+y+z=0$.  Solving for the short roots where $V(t)$ is a maximum can be done in either representation, and these are equivalent.
  \end{remark}
\bigskip
\noindent For now, let’s set aside the question of how much about an element of $T/W$ is determined by the cosines of the short roots of the log, and proceed with the solution of the following problem:\\

 \noindent Let $\alpha = \cos\theta_3$,  $\beta = \cos\theta_2$, $\gamma = \cos\theta_1$.  Let $A = \alpha + \beta + \gamma$, $B = \alpha\beta + \beta\gamma + \alpha\gamma$, and $C = \alpha\beta\gamma$ be the elementary symmetric functions of $\alpha$, $\beta$, $\gamma$ so that 
$$f(x) = (x-\alpha)(x-\beta)(x-\gamma) = x^3 -Ax^2 + Bx - C $$
vanishes at $\alpha$, $\beta$, $\gamma$.  Reordering the $\theta_i$s if necessary, we may assume $\alpha < \beta < \gamma$, so that
$D(\alpha,\beta,\gamma) = (\beta-\alpha)(\gamma-\beta)(\gamma-\alpha)$ is the positive square root of the discriminant of $f$.  Lemma  \ref{diff of cosines} shows that if $t = \exp_T(\theta_1,\theta_2,\theta_3)$ maximizes $V(t)$, then $(\alpha, \beta, \gamma)$ maximizes $D$ subject to the constraint $\theta_1+\theta_2+\theta_3=0$.

\begin{proposition}  \label{A=B} With notation as above, let $$\rho(\alpha, \beta, \gamma) = -\cos^{-1}(\alpha) + \cos^{-1}(\beta) + \cos^{-1}(\gamma).$$  At the maximum of $D(\alpha,\beta,\gamma)$ subject to $\rho(\alpha, \beta, \gamma) = 0$, the equality $A = B$ holds, i.e.
$$ \alpha+\beta+\gamma = \alpha\beta + \beta\gamma + \alpha\gamma.$$
Here the $\pm$ ambiguity in $\cos^{-1}$ is resolved by choosing the standard value between $0$ and $\pi$ for $\beta$ and $\gamma$ and the negative of the standard value for $\alpha$, and that’s why there is a minus sign on the first term of $\rho$.
\end{proposition}

\begin{proof} By Lagrange multipliers, for $(\alpha, \beta, \gamma)=\mathrm{argmax}\{ D \mid \rho=0\}$, the gradients of $D$ and $\rho$ are aligned, i.e.\ $\nabla D = \lambda\nabla\rho$ for some constant $\lambda$.  Because $D$ is a translation invariant function of 
$(\alpha, \beta, \gamma)$, we know that $\nabla D$ is orthogonal to $(1,1,1)$, and thus $\nabla\rho$ must also be orthogonal to $(1,1,1)$.  Equivalently,
$$\frac{-1}{\sqrt{1-\alpha^2}} + \frac{1}{\sqrt{1-\beta^2}} + \frac{1}{\sqrt{1-\gamma^2}} = 0,$$ or
$$\frac{-1}{\sin(\theta_1+\theta_2)} + \frac{1}{\sin\theta_2} + \frac{1}{\sin\theta_1} = 0.$$
\bigskip
\noindent Clearing fractions gives
\begin{equation} \label{vanishing sine formula}
-\sin\theta_1\sin\theta_2 + \sin\theta_1\sin(\theta_1+\theta_2) + \sin\theta_2\sin(\theta_1+\theta_2) = 0.
\end{equation}  

\noindent Using the identity $\sin x\sin y = \frac{1}{2}(\cos(x-y) - \cos(x+y))$ we may replace each product of sines above with a difference of cosines.  The result is 
$$\cos(\theta_1+\theta_2) - \cos(\theta_1-\theta_2) + \cos(\theta_2) - \cos(2\theta_1+\theta_2) +  \cos(\theta_1) - \cos(\theta_1+2\theta_2) = 0,$$  or
\begin{align} \label{A = cosines}
A & = \alpha + \beta + \gamma   = \cos\theta_1 + \cos\theta_2 + \cos(\theta_1+\theta_2)  \notag \\
   &= \cos(\theta_1-\theta_2) + \cos(2\theta_1+\theta_2) + \cos(\theta_1+2\theta_2).
\end{align}
Expanding the RHS of \eqref{A = cosines} and separating into terms involving cosines followed by terms involving sines gives 
\begin{align*} A  = & \cos\theta_1\cos\theta_2 + \cos(\theta_1+\theta_2)\cos\theta_1 + \cos(\theta_1+\theta_2)\cos\theta_2 \\ 
                              & - [ -\sin\theta_1\sin\theta_2 + \sin\theta_1\sin(\theta_1+\theta_2) + \sin\theta_2\sin(\theta_1+\theta_2)].
                              \end{align*}
By \eqref{vanishing sine formula} the expression in the brackets vanishes, and the expression involving cosines is just $B$, and thus $A = B$.
\end{proof}

\begin{theorem}  \label{poly for LM} With notation as above, let 
$$g(x) = \left(x - \frac{A}{3}\right)^2(-3x^2 + 2Ax + A^2 - 4B)(1 - x^2).$$
Then if  $D(\alpha,\beta,\gamma)$ is a maximum subject to $\rho(\alpha, \beta, \gamma) = 0$, there is a constant $d$ for which 
\begin{equation}\label{eqn for poly for LM}
g(\alpha) = g(\beta) = g(\gamma) = d,
\end{equation}i.e.\ the roots of $g(x) - d = 0$ include $\alpha$, $\beta$, and $\gamma$. That is, $g(x) - d$ is divisible by $f(x) = x^3 - Ax^2 + Bx - C$.
\end{theorem}

\begin{proof}  We will derive \eqref{eqn for poly for LM} from the 3 components of the equality $\nabla D = \lambda\nabla\rho$. First observe that
\[3\alpha - A = (\alpha-\beta) + (\alpha-\gamma).\]
Thus
\[\frac{\partial D}{\partial\alpha} = (\gamma-\beta)[  (\alpha-\beta) + (\alpha-\gamma) ] = (\gamma-\beta)(3\alpha - A).\]
The idea behind the next step is to eliminate $\gamma-\beta$ from (the square of) the previous equation, by expressing $(\gamma-\beta)^2$ in terms of $\alpha, A, B$. It is straightforward to find the required identity:
$$(\gamma-\beta)^2 =  -3\alpha^2 + 2A\alpha + A^2 - 4B.$$

\noindent Thus
$$\left(\frac{\partial D}{\partial\alpha} \right)^2 = (3\alpha-A)^2(-3\alpha^2 + 2A\alpha + A^2 - 4B).$$
 On the other hand 
$\frac{\partial \rho}{\partial\alpha} = \frac{-1}{\sqrt{1-\alpha^2}}$, so $(\frac{\partial \rho}{\partial\alpha} )^2 = \frac{1}{1-\alpha^2}$. So the first component of $\nabla D = \lambda\nabla\rho$ implies
$$(3\alpha - A)^2(-3\alpha^2 + 2A\alpha + A^2 - 4B)(1 - \alpha^2) = \lambda^2.$$
Similarly, the other two components of $\nabla D = \lambda\nabla\rho$ can be written as 
\begin{align*}
(3\beta - A)^2(-3\beta^2 + 2A\beta + A^2 - 4B)(1 - \beta^2) &= \lambda^2 \\
(3\gamma - A)^2(-3\gamma^2 + 2A\gamma + A^2 - 4B)(1 - \gamma^2) &= \lambda^2
\end{align*}
This proves \eqref{eqn for poly for LM}, using $d=\lambda^2/9$.
\end{proof}

\begin{theorem}  \label{three-fifths} At a maximum of $V(t)$, the cosines of the three short roots evaluated at $\log(t)$ are the roots of the cubic equation $x^3 + \frac{3}{5}x^2 - \frac{3}{5}x - \frac{7}{25} = 0$.
\end{theorem}

\begin{proof}  
Using Proposition \ref{A=B}, we can replace $B$ by $A$ and write the monic polynomial $g(x)/3$ as 
\begin{align*}
\frac13 g(x) &= x^6 - \frac{4A}{3}x^5 +  \frac{2A^2+12A-9}{9}x^4 + \frac{4A^3-24A^2+36A}{27}x^3 \\
& + \frac{-A^4+4A^3-6A^2-36A}{27}x^2 + \frac{-4A^3+24A^2}{27}x + \frac{A^4-4A^3}{27}
\end{align*}
By Theorem \ref{poly for LM},  the remainder when dividing $\frac13 g(x)$ by $f(x)$ has degree 0.
 Computing the coefficients of $x^2$ and $x$ in this remainder yields the following two equations.
\begin{equation} \label{eq1 for C}
A^3 - 6A^2 - 9A + 18AC = 0
\end{equation}
\begin{equation} \label{eq2 for C}
-A^4 + 2A^3 + 15A^2 - (3A^2 + 18A + 27)C = 0
\end{equation}
 Equating the two resulting expressions for $C$ leads to 
\begin{equation} \label{deg4 for A}
5A^4 - 12A^3 - 54A^2 + 108A + 81 = 0
\end{equation}
The roots of \eqref{deg4 for A} are $-\frac{3}{5}$ and -3 with multiplicity 1 and 3 with multiplicity 2.  Note that 3 and -3 are the maximum and minimum possible values for the sum of three cosines, so these correspond to minima of $V(t)$ leaving $A = -\frac{3}{5}$  as the only possibility for the maximum of $V(t)$. Finally setting   $A = -\frac{3}{5}$ in \eqref{eq1 for C} yields $C = \frac{7}{25}$.
\end{proof}
\begin{remark}
The cosines of the rotation angles of the largest conjugacy class in $SO(7)$ are the roots of $x^3 +\frac{3}{5}x^2 - \frac{3}{5}x - \frac{1}{5} = 0$.  That equation differs only by a constant from the one established here for the cosines of the short roots of the largest conjugacy class of $\mathrm G_2$.  Since $\mathrm G_2$ is a subgroup of $SO(7)$, maybe there is an easier proof for Theorem \ref{three-fifths}.
\end{remark}

\begin{corollary}
For $\mathrm G_2$ there is a unique maximum for $V(t)$ in $T/W$.
\end{corollary}
\begin{proof}
We claim that there is a unique choice of $(\theta_1,\theta_2,\theta_3)$ in $\lie t/W$ subject to the constraints \\
(1) $\{\alpha = \cos\theta_1, \beta = \cos\theta_2, \gamma = \cos\theta_3\}$ are the three roots of \\
$f(x) = x^3 + \frac{3}{5}x^2 - \frac{3}{5}x - \frac{7}{25} = 0$ \\
and \\
(2) $\theta_1 + \theta_2 + \theta_3 = 0$. \\
We know there is at least one such choice because there must be a maximum for $V(t)$, and the log of any such maximum must satisfy these constraints.  Any two such choices must differ by sign changes.  If the number of sign changes is one or two, then at least one element of $\{\theta_1,\theta_2,\theta_3\}$ must be 0, and this would lead to \\
$V(\exp_T(\theta_1,\theta_2,\theta_3)) = 0$.  But if the number of sign changes is three, the two choices differ by an element of the Weyl group, because the composition of the reflections corresponding to any 2 orthogonal roots is $-\rm{Id}$.
\end{proof}

\section{Appendix}

In this Appendix we present a Theorem and proof based on the Sturm Comparison Theorem equivalent to Corollary \ref{equidistribution}.

\begin{theorem} \label{LimitingDensityOfRoots}
  Let $k(x)$ be continuous with $0 < \beta < k(x) < \alpha$ on the interval $[x_0-a,x_0+a]$.  Let $v_{\lambda,k}(x)$ be the unique solution of $v_{\lambda,k}''(x) = -\lambda^2{k(x)}^2v_{\lambda,k}(x)$ with $v_{\lambda,k}(x_0) = 1$ and $v_{\lambda,k}'(x_0) = 0$.  For any $0 < \delta < a$, let $N_{\lambda,k,\delta}$ be the number of solutions of $v_{\lambda,k}(x) = 0$ in the interval $[x_0-\delta,x_0+\delta]$.  Define $\rho_k(x_0) = \displaystyle{\lim_{\delta\to0}\lim_{\lambda\to\infty}\frac{N_{\lambda,k,\delta}}{2\lambda\delta}}$.  Then $\rho_k(x_0) = \frac{k(x_0)}{\pi}$.
\end{theorem}

\noindent [To apply this Theorem in the context of the proof of Corollary \ref{equidistribution}, use $\lambda = \sqrt{n(n-1)}$, $k(x) = 1/\sqrt{(1-x^2)}$, and $\Delta x = 2\delta$, for any $x_0 \in (-1,1)$.]

\bigskip
\noindent Our proof of Theorem \ref{LimitingDensityOfRoots} is based on the Sturm Comparison Theorem [see \cite{SturmWiki}, \cite{Sturm'sTheorems}].  Here is a simple statement,adequate for our purposes.
\bigskip

\noindent Sturm Comparison Theorem [SCT] (from \cite{Sturm'sTheorems}, for a more general statement see \cite{SturmWiki}):

\bigskip
 \noindent Let $\phi_1$ and $\phi_2$ be non-trivial solutions of equations
  $$y'' + q_1(x)y = 0$$
  and
  $$y'' + q_2(x)y = 0$$
respectively, on an interval $I$ where $q_1$ and $q_2$ are continuous functions such that $q_1(x) \leq q_2(x)$ on $I$.  Then between any two consecutive zeroes $x_1$ and $x_2$ of $\phi_1$, there exists at least one zero of $\phi_2$ unless $q_1(x) \equiv q_2(x)$ on $(x_1, x_2)$.

\bigskip
\noindent The proof given in \cite{Sturm'sTheorems} is based on analyzing the integral of $(\phi_1\phi_2' - \phi_2\phi_1')'$ over the interval $[x_1,x_2]$.  Another way of stating the result of SCT is that the number of zeros of $\phi_2$ in $I$ is greater than or equal to the number of intervals between consecutive zeros of $\phi_1$ in $I$.

\bigskip
\begin{proposition} \label{NumberOfRoots}
Under the assumptions of Theorem \ref{LimitingDensityOfRoots}, $$\frac{2\delta\lambda\beta}{\pi} - 1 \leq N_{\lambda,k,\delta} \leq \frac{2\delta\lambda\alpha}{\pi} + 2$$.
\end{proposition}

\begin{proof}
Let $u_{\lambda,\beta}(x) = sin\lambda\beta(x - x_0)$ be the solution of $u_{\lambda,\beta}'' +  \lambda^2\beta^2u_{\lambda,\beta} = 0$ with initial conditions $u_{\lambda\beta}(x_0) = 0$ and $u_{\lambda,\beta}'(x_0) = \lambda\beta$.  Similarly, let $u_{\lambda,\alpha}(x) = sin\lambda\alpha(x - x_0)$ be the solution of $u_{\lambda,\alpha}'' +  \lambda^2\alpha^2u_{\lambda,\alpha} = 0$ with initial conditions $u_{\lambda\alpha}(x_0) = 0$ and $u_{\lambda,\alpha}'(x_0) = \lambda\alpha$.

\bigskip
\noindent First apply SCT to $\phi_1 = u_{\lambda,\beta}$ and $\phi_2 = v_{\lambda,k}$ on the interval $I = [x_0 - \delta,x_0 + \delta]$.  Each interval between consecutive zeros of $u_{\lambda,\beta}$ has length $\frac{\pi}{\lambda\beta}$ so the number of intervals between consecutive zeros of $\phi_1$ in $I$ is $\left\lfloor\frac{2\delta\lambda\beta}{\pi}\right\rfloor \geq \frac{2\delta\lambda\beta}{\pi} - 1$.  Thus $$\frac{2\delta\lambda\beta}{\pi} - 1 \leq N_{\lambda,k,\delta}$$.

\bigskip
\noindent Second apply SCT to $\phi_2 = u_{\lambda,\alpha}$ and $\phi_1 = v_{\lambda,k}$ on the interval $I$.  The number of intervals between consecutive zeros of $\phi_1$ in $I$ is $N_{\lambda,k,\delta} - 1$.  The number of zeros of $\phi_2$ in $I$ is $\left\lfloor\frac{2\delta\lambda\alpha}{\pi}\right\rfloor + 1 \leq \frac{2\delta\lambda\alpha}{\pi} + 1$.  Thus $$N_{\lambda,k,\delta} \leq \frac{2\delta\lambda\alpha}{\pi} + 2$$.
\end{proof}

\begin{proof} [Proof of Theorem \ref{LimitingDensityOfRoots}]
By Proposition \ref{NumberOfRoots} $$\frac{\beta}{\pi} - \frac{1}{2\lambda\delta} \leq \frac{N_{\lambda,k,\delta}}{2\lambda\delta} \leq \frac{\alpha}{\pi} + \frac{2}{2\lambda\delta}.$$  Taking the limit as $\lambda\to\infty$ gives $$\frac{\beta}{\pi}  \leq \displaystyle{\lim_{\lambda\to\infty}\frac{N_{\lambda,k,\delta}}{2\lambda\delta}} \leq \frac{\alpha}{\pi}.$$  Since k(x) is continuous, given $\epsilon > 0$ we can choose $\delta > 0$ so that $$k(x_0) - \epsilon < k(x) < k(x_0) + \epsilon$$ on $[x_0 - \delta, x_0 + \delta]$.  Therefore, for any $\epsilon > 0$, for sufficiently small $\delta > 0$, we have $$\frac{k(x_0) - \epsilon}{\pi}  \leq \displaystyle{\lim_{\lambda\to\infty}\frac{N_{\lambda,k,\delta}}{2\lambda\delta}} \leq \frac{k(x_0) + \epsilon}{\pi}.$$ Now taking limits as $\epsilon\to0$ gives the desired result.
\end{proof}

\section{Acknowledgments}

Thanks to reviewers Sam Lichtenstein, Keith Conrad, Joe Wolf, and Paul Howard.

\end{document}